\documentclass[12pt,a4paper]{amsart}
\usepackage[latin1]{inputenc}
\usepackage[T1]{fontenc}
\usepackage[right=3cm,left=3cm,top=3cm,bottom=3cm]{geometry}
\usepackage{ifthen}
\usepackage{tabularx}
\usepackage[all]{xy}
\usepackage{pstricks,pst-plot,pst-text,pst-tree}
\usepackage{pst-eps,pst-fill,pst-node,pst-math} 
\usepackage{multido}
\usepackage{enumerate}
\usepackage{amsmath}
\usepackage{amsthm}
\usepackage{amsfonts}
\usepackage{amssymb} 
\usepackage{ae}
\usepackage[dvips]{graphicx}
\usepackage[np]{numprint}
\usepackage{slashbox}
\usepackage{textcomp}
\usepackage[hang]{footmisc}
\usepackage[normalem]{ulem}
\usepackage{hyperref} 

\npdecimalsign{\ensuremath{.}}

\title{From the Ideal Theorem to the class number}
\author{Olivier Bordell\`{e}s}
\date{}

\newcommand{\Z}{\mathbb {Z}}

\theoremstyle{theorem}
\newtheorem{theorem}{Theorem}
\newtheorem{lem}[theorem]{Lemma}
\newtheorem{coro}[theorem]{Corollary}

\theoremstyle{definition}

\theoremstyle{remark}

\makeatletter
\@namedef{subjclassname@2020}{%
  \textup{2020} Mathematics Subject Classification}
\makeatother

\begin{document}

\begin{abstract}
In this note, we provide an explicit upper bound for $h_K \mathcal{R}_K d_K^{-1/2}$ which depends on an effective constant in the error term of the Ideal Theorem.
\end{abstract}

\subjclass[2020]{11R29, 11R42}
\keywords{Class number, Dedekind zeta-function, Ideal Theorem.}

\maketitle

\section{Introduction}

Let $K$ be a number field of degree $n \geqslant 3$, signature $(r_1,r_2)$, discriminant $(-1)^{r_2}d_K$, class number $h_K$, regulator $\mathcal{R}_K$ and let $w_K$ be the number of roots of unity in $K$. Let $\kappa_K$ be the residue at $s=1$ of the Dedekind zeta-function $\zeta_K(s)$ attached to $K$. \\

Estimating $h_K$ is a long-standing problem in algebraic number theory. One of the classic way is the use of the so-called \textit{analytic class number formula} stating that
\begin{equation}
   h_K \mathcal{R}_K = \frac{w_K }{2^n} \left( \frac{2}{\pi} \right)^{r_2} d_K^{1/2} \, \kappa_K \label{e1}
\end{equation}
and to use Hecke's integral representation and the functional equation of the Dedekind zeta function to majorize $\kappa_K$. This is done in \cite{lou1,lou2} with additional properties of log-convexity of some functions related to $\zeta_K$ which enables Louboutin to reach the following bound
\begin{equation}
   h_K \mathcal{R}_K \leqslant \frac{w_K}{2} \left( \frac{2}{\pi} \right)^{r_2} \left( \frac{e \log d_K}{4n-4} \right)^{n-1} d_K^{1/2}. \label{eq:Louboutin}
\end{equation}

Let $r_K(m)$ be the $m$-th coefficient of $\zeta_K$, i.e., the number of non-zero integral ideals of $\mathcal{O}_K$ of norm $m$, and denote $\Delta_K(x)$ to be the error term in the Ideal Theorem, i.e.,
\begin{equation}
   \Delta_K(x) = \sum_{m \leqslant x} r_K(m) - \kappa_K x. \label{e2}
\end{equation}

The aim of this work is to prove the following result.

\begin{theorem}
\label{t1}
Let $K$ be an algebraic number field of degree $n \geqslant 3$ and set $\gamma_n = 214$ if $n=3$, $\gamma_n = 10$ if $n \geqslant 4$. Assume that there exist $\alpha \in \left( 0,\frac{2}{n} \right)$ and a constant $C_K \geqslant \exp \left( \max \left( \gamma_n \, , \, \alpha n + \frac{1}{4 \alpha^2}\right) \right)$ such that for $x \geqslant 1$
\begin{equation}
   \left | \Delta_K(x) \right | \leqslant C_K x^{1-\alpha} \label{e3}
\end{equation}
where $\Delta_K(x)$ is given in \eqref{e2}. Then
$$h_K \mathcal{R}_K < \frac{3w_K}{2} \left( \frac{2}{\pi} \right)^{r_2} \left( \frac{\left( \frac{1}{2 \alpha} \log C_K \right)^{n-1}}{(n-1)!} - \frac{\left( \frac{1}{2 \alpha} \log C_K \right)^{n-2}}{(n-2)!} \right) d_K^{1/2}.$$
\end{theorem}

\section{Tools}

The first lemma is a Titchmarsh like generalization of \cite[Theorem~12.5]{tit} to number fields established by Ayoub \cite{ayo}. The result is stated for the quadratic case, but as it can be seen in the proof and as the author points it out, it is still true for the general case.

\begin{lem}
\label{le1}
Let $n \geqslant 3$ and $\mu_K$ be the infimum of the real numbers $\sigma$ for which the integral
$$\int_{- \infty}^\infty \frac{| \zeta_K ( \sigma + it) |^2}{| \sigma + it |^2} \, \mathrm{d}t$$
converges. Then $\mu_K \leqslant 1 - \frac{2}{n}$ and, for all $\mu_K < \sigma < 1$, we have
$$\frac{1}{2 \pi} \int_{- \infty}^\infty \frac{| \zeta_K ( \sigma + it) |^2}{| \sigma + it |^2} \, \mathrm{d}t = \int_0^\infty \Delta_K(x)^2 x^{-1-2\sigma} \, \mathrm{d}x.$$
\end{lem}

\begin{proof}
For the first part of the lemma, it suffices to show that
$$\int_{T}^\infty \frac{| \zeta_K ( \sigma + it) |^2}{| \sigma + it |^2} \, \mathrm{d}t$$
tends to zero as $T \to \infty$ uniformly for $\sigma > 1 - \frac{2}{n}$. Splitting the interval $[T, \infty)$ into dyadic subintervals of the shape $[2^k T, 2^{k+1}T)$ ($k \geqslant 0$), we see that the above result is a consequence of an estimate of the form
$$\int_{T/2}^T \frac{| \zeta_K ( \sigma + it) |^2}{| \sigma + it |^2} \, \mathrm{d}t \ll T^{- \eta}$$
for some $\eta  > 0$. But by \cite[Theorem~4]{cha}, we have uniformly for $\sigma \geqslant 1 - \frac{2}{n} + \eta$
$$\int_{1}^T | \zeta_K ( \sigma + it) |^2 \, \textrm{d}t  \ll \left( T + T^{n(1-\sigma)} \right) (\log T)^n $$
implying the desired bound. The second part of the lemma follows from Parseval's formula for the Mellin's transform.
\end{proof}

\begin{coro}
\label{co1}
Assume hypothesis \eqref{e3}, and let $0 < \delta < \alpha < \frac{2}{n}$. Then
$$\frac{1}{2 \pi} \int_{- \infty}^\infty \frac{| \zeta_K ( 1 - \delta + it) |}{| 1-\delta + it |^2} \, \mathrm{d}t \leqslant \frac{1}{2\sqrt{1-\delta}} \left( \frac{\kappa_K}{\sqrt{\delta}} + \frac{C_K}{\sqrt{\alpha-\delta}} \right).$$
\end{coro}

\begin{proof}
Using Lemma~\ref{le1}, \eqref{e3} and the trivial bound $\left | \Delta_K(x) \right | \leqslant \kappa_K x$ when $x \in \left[ 0,1 \right)$
\begin{align*}
   \frac{1}{2 \pi} \int_{- \infty}^\infty \frac{| \zeta_K ( 1 - \delta + it) |^2}{| 1-\delta + it |^2} \, \textrm{d}t & \leqslant  \kappa_K^2 \int_0^1 x^{2 \delta - 1} \textrm{d}x + C_K^2 \int_1^\infty x^{2 (\delta-\alpha) - 1 } \textrm{d}x \\
   & =  \frac{1}{2} \left( \frac{\kappa_K^2}{\delta} + \frac{C_K^2}{\alpha-\delta} \right) 
\end{align*}
and using Cauchy-Schwarz's inequality we get
\begin{align*}
   &  \frac{1}{2 \pi} \int_{- \infty}^\infty \frac{| \zeta_K ( 1 - \delta + it) |}{| 1-\delta + it |^2} \, \mathrm{d}t \\
   & \quad \leqslant  \frac{1}{2 \pi} \left( \int_{- \infty}^\infty \frac{| \zeta_K ( 1 - \delta + it) |^2}{| 1-\delta + it |^2} \, \textrm{d}t \right)^{1/2} \left( \int_{- \infty}^\infty \frac{\textrm{d}t}{ |1-\delta + it |^2} \right)^{1/2} \\
   & \quad \quad \leqslant  \frac{1}{2 \sqrt{1 - \delta}} \left( \frac{\kappa_K^2}{\delta} + \frac{C_K^2}{\alpha - \delta} \right)^{1/2} \leqslant  \frac{1}{2\sqrt{1-\delta}} \left( \frac{\kappa_K}{\sqrt{\delta}} + \frac{C_K}{\sqrt{\alpha-\delta}} \right)
\end{align*}
 as asserted.
\end{proof}

\begin{lem}
\label{le:inequality_weighted_log}
Uniformly for all $x \geqslant 1$ and all $n \in \Z_{\geqslant 1}$, we have
$$\sum_{m \leqslant x} r_K (m) \log \frac{x}{m} \leqslant \frac{x \log x}{(n-1)!} \left( \log x+ n -1\right)^{n-2}.$$
\end{lem}

\begin{proof}
Let $\tau_n$ be the $n$th Piltz-Dirichlet divisor function. We have $r_K (m) \leqslant \tau_n(m)$ and from the bound \cite[Exercise~78]{bor} 
$$\sum_{m \leqslant t} \tau_n(m) \leqslant t \sum_{j=0}^{n-1} \binom{n-1}{j} \frac{(\log t)^j}{j!} \quad \left( t \geqslant 1 \right)$$
so that
\begin{align*}
   \sum_{m \leqslant x} r_K (m) \log \frac{x}{m} & =  \int_1^x \frac{1}{t} \left( \sum_{m \leqslant t} r_K(m) \right) \, \textrm{d}t \\
   & \leqslant  \sum_{j=0}^{n-1} \binom{n-1}{j} \frac{1}{j!} \int_1^x (\log t)^j \, \textrm{d}t \\
   &= x \sum_{j=0}^{n-1} \binom{n-1}{j} (-1)^j \left\lbrace \sum_{k=0}^j \frac{(- \log x)^k}{k!} - \frac{1}{x} \right\rbrace  \\
   &= - x \sum_{j=0}^{n-1} \binom{n-1}{j} (-1)^j \sum_{k=j+1}^\infty \frac{(- \log x)^k}{k!} \\
   &= - x \sum_{k=1}^\infty \frac{(- \log x)^k}{k!} \sum_{j=0}^{\min(n-1,k-1)} (-1)^j \binom{n-1}{j} \\*
   &= -x \sum_{k=1}^\infty (-1)^{k+\min(n-1,k-1)}\frac{(\log x)^k}{k!} \binom{n-2}{\min(n-1,k-1)} \\
   &= x \sum_{k=1}^{n-1} \frac{(\log x)^k}{k!} \binom{n-2}{k-1} \\
   & \leqslant  \frac{x \log x}{(n-1)!} \left( \log x+ n -1\right)^{n-2}
\end{align*}
where we used the fact that, for $0 \leqslant k \leqslant n-3$
\begin{align*}
   \frac{1}{(k+1)!} &= \frac{1}{(n-1)!} \prod_{i=2}^{n-k-1} (i+k) \\
   & \leqslant  \frac{1}{(n-1)!} \left( \frac{1}{n-k-2} \sum_{i=2}^{n-k-1} (i+k) \right)^{n-k-2} \\
   & \leqslant  \frac{(n-1)^{n-k-2}}{(n-1)!}
\end{align*}
by the geometric-arithmetic mean inequality, the bound still being valid in the case $k= n-2$.
\end{proof}

\begin{lem}
\label{le:tech}
If $n \geqslant 3$, $0 < \alpha < \frac{2}{n}$ and $\log C_K \geqslant \gamma_n$, then 
$$\frac{\left( \alpha^{-1} \log C_K  \right)^{n-2}}{(n-2)!} \geqslant \frac{e^{\alpha(n-1)+1} \sqrt{\log C_K}}{\sqrt{\alpha(1-\alpha)}}.$$
\end{lem}

\begin{proof}
Squaring the inequality of the lemma, it is equivalent to show
$$(\log C_K)^{2n-5} \geqslant ((n-2)!)^2 \times \frac{\alpha^{2n-5} e^{2\alpha(n-1)+2}}{1-\alpha}.$$
The function $\alpha \in \left( 0, \frac{2}{n} \right) \mapsto \frac{\alpha^{2n-5} e^{2\alpha(n-1)+2}}{1-\alpha}$ is non-decreasing, so that
$$\frac{\alpha^{2n-5} e^{2\alpha(n-1)+2}}{1-\alpha} \leqslant \frac{2^{2n-5}e^{6-4/n}n^{6-2n}}{n-2}$$
and therefore it suffices to show
$$\log C_K \geqslant 2 \left( \frac{e^{6-4/n}n^{6-2n}}{n-2} \right)^{\frac{1}{2n-5}} ((n-2)!)^{\frac{2}{2n-5}}.$$
Using Stirling's bound, the inequality of the lemma is guaranteed as soon as
$$\log C_K \geqslant 2 (2 \pi)^{\frac{1}{2n-5}} e^{- \frac{12n^3-84n^2+143n-48}{6n(n-2)(2n-5)}} \left( \frac{(n-2)^{2n-4}}{n^{2n-6}}\right)^{\frac{1}{2n-5}}$$
and since the right-hand side is a non-increasing function in $n \geqslant 3$, it then suffices that $\log C_K \geqslant s_n$ where $s_3 := 4 \pi e^{17/6} \approx \np{213.7}$ and $s_n = 2^{4/3} \pi^{1/3} e^{13/36} \approx \np{5.3}$ whenever $n \geqslant 4$.
\end{proof}

\section{Proof of the main result}

By \eqref{e1} it is sufficient to show that
\begin{equation}
   \kappa_K < 3 \left( \frac{\left( \alpha^{-1} \log C_K  \right)^{n-1}}{(n-1)!} - \frac{2\left( \alpha^{-1} \log C_K  \right)^{n-2}}{(n-2)!} \right). \label{e4}
\end{equation}

Assume $n \geqslant 3$, and let $0 < \delta < \alpha < \frac{2}{n}$ and $x \geqslant 1$ satisfying
\begin{equation}
   \delta (1-\delta) \geqslant x^{-2 \delta} \label{eq:condition_1}
\end{equation}
and
\begin{equation}
x \geqslant e^{1 + \frac{1}{4 \alpha^3}}. \label{eq:condition_2}
\end{equation}
By Perron's formula (see \cite[(2.9) page~220]{ten} for instance) we have
$$\sum_{m \leqslant x} r_K (m) \log \frac{x}{m} = \frac{1}{2 \pi i} \int_{2-i \infty}^{2+i \infty} \frac{\zeta_K(s)}{s^2} \, x^s \, \textrm{d}s.$$
Shifting the contour integration to the line $\sigma = 1 - \delta$ and picking up the residue of the integrand at the unique simple pole $s=1$, we obtain by Cauchy's theorem
\begin{align*}
   \sum_{m \leqslant x} r_K (m) \log \frac{x}{m} & =  \kappa_K x + \frac{1}{2 \pi i} \int_{1 - \delta-i \infty}^{1-\delta+i \infty} \frac{\zeta_K(s)}{s^2} \, x^s \, \textrm{d}s \\
   & :=  \kappa_K x + I_\delta (x)
\end{align*} 
and using Corollary~\ref{co1} we get
\begin{align*}
   \left | I_\delta (x) \right | & \leqslant  \frac{x^{1- \delta}}{2 \pi} \int_{- \infty}^\infty \frac{| \zeta_K ( 1 - \delta + it) |}{| 1-\delta + it |^2} \, \mathrm{d}t \\
   & \leqslant  \frac{x^{1-\delta}}{2\sqrt{1-\delta}} \left( \frac{\kappa_K}{\sqrt{\delta}} + \frac{C_K}{\sqrt{\alpha-\delta}} \right).
\end{align*}
Therefore, using \eqref{eq:condition_1}, we derive
\begin{align*}
   \sum_{m \leqslant x} r_K (m) \log \frac{x}{m} & \geqslant \kappa_K x \left( 1-\frac{x^{-\delta}}{2\sqrt{\delta (1-\delta)}}\right) -\frac{x^{1-\delta}}{2} \frac{C_K}{\sqrt{(1-\delta)(\alpha-\delta)}} \\
   & \geqslant \frac{\kappa_K x}{2}  -\frac{x^{1-\delta}}{2} \frac{C_K}{\sqrt{(1-\delta)(\alpha-\delta)}}
\end{align*}
and Lemma~\ref{le:inequality_weighted_log} yields
\begin{align*}
   \kappa_K & \leqslant \frac{2}{x} \sum_{m \leqslant x} r_K (m) \log \frac{x}{m} + \frac{x^{-\delta}C_K}{\sqrt{(1-\delta)(\alpha-\delta)}} \\
   & \leqslant \frac{2 \log x}{(n-1)!} \left( \log x+ n -1\right)^{n-2} + \frac{x^{-\delta}C_K}{\sqrt{(1-\delta)(\alpha-\delta)}}    
\end{align*}
whenever $x$ and $\delta$ satisfy \eqref{eq:condition_1}. Now choose
\begin{equation}
   \delta = \alpha - \frac{1}{\log x}. \label{eq:choice_delta}
\end{equation}
Note that \eqref{eq:condition_1} is satisfied if $\left( \alpha \log x - 1 \right) \left(1 + (1-\alpha) \log x \right) \geqslant e^2 (\log x)^2 x^{-2\alpha}$. But using \eqref{eq:condition_2} we get
$$\left( \alpha \log x - 1 \right) \left(1 + (1-\alpha) \log x \right) \geqslant \frac{1}{16 \alpha^2} \geqslant e^2 (\log x)^2 x^{-2\alpha}.$$
Therefore, with the choice \eqref{eq:choice_delta}, we derive
\begin{align*}
   \kappa_K & \leqslant \frac{2 \log x}{(n-1)!} \left( \log x+ n -1\right)^{n-2} + \frac{e \, C_K x^{-\alpha} \log x}{\sqrt{(1-\alpha) \log x + 1}} \\
   & <     \frac{2 \log x}{(n-1)!} \left( \log x+ n -1\right)^{n-2} + \frac{e \, C_K x^{-\alpha} }{\sqrt{1-\alpha}} \sqrt{\log x}
\end{align*}
provided that \eqref{eq:condition_2} is fullfilled. We next choose
\begin{equation}
   x = C_K^{1/\alpha} e^{1-n}. \label{eq:choice_x}
\end{equation}
This yields
\begin{align*}
   \kappa_K & < 2 \left( \frac{\left( \alpha^{-1} \log C_K  \right)^{n-1}}{(n-1)!} - \frac{\left( \alpha^{-1} \log C_K  \right)^{n-2}}{(n-2)!} \right) + \frac{e^{\alpha(n-1)+1} \sqrt{\log C_K + \alpha(1-n)}}{\sqrt{\alpha(1-\alpha)}} \\
   & < 2 \left( \frac{\left( \alpha^{-1} \log C_K  \right)^{n-1}}{(n-1)!} - \frac{\left( \alpha^{-1} \log C_K  \right)^{n-2}}{(n-2)!} \right) + \frac{e^{\alpha(n-1)+1} \sqrt{\log C_K}}{\sqrt{\alpha(1-\alpha)}}.
\end{align*}
Now, since $\log C_K \geqslant \gamma_n \geqslant 10 > 10 - \frac{10}{n} \geqslant 5 \alpha(n-1)$, we derive using Lemma~\ref{le:tech}
\begin{align*}
   \frac{\left( \alpha^{-1} \log C_K  \right)^{n-1}}{(n-1)!} - \frac{4\left( \alpha^{-1} \log C_K  \right)^{n-2}}{(n-2)!} &= \frac{\left( \alpha^{-1} \log C_K  \right)^{n-2}}{(n-2)!} \left( \frac{\alpha^{-1}\log C_K}{n-1} - 4 \right) \\
   & > \frac{\left( \alpha^{-1} \log C_K  \right)^{n-2}}{(n-2)!} \\
   & \geqslant \frac{e^{\alpha(n-1)+1} \sqrt{\log C_K}}{\sqrt{\alpha(1-\alpha)}}
\end{align*}
and therefore
$$\kappa_K < 3 \left( \frac{\left( \alpha^{-1} \log C_K  \right)^{n-1}}{(n-1)!} - \frac{2\left( \alpha^{-1} \log C_K  \right)^{n-2}}{(n-2)!} \right)$$
as required.
\qed

\section{Example}

Improving a result in Sunley's thesis \cite{sun}, Lee \cite{lee22} proved that, for all $x > 0$
$$\left | \Delta_K(x) \right | \leqslant \Theta_K \, d_K^{\frac{1}{n+1}} (\log d_K)^{n-1} x^{1-\frac{2}{n+1}}$$
with $\Theta_K := \np{0.17} \left( \frac{6n-2}{n-1} \right) \np{2.26}^n \, e^{4n+26/n} \, n^{n+1/2} \left( \np{44.39} \times \np{0.082}^n \, n! + \frac{13}{n-1} \right)$. Hence one can take $\alpha = \frac{2}{n+1}$ and
$$C_K := \Theta_K \, d_K^{\frac{1}{n+1}} (\log d_K)^{n-1}.$$
The hypothesis $\log C_K \geqslant \max \left( \gamma_n \, , \, \alpha n + \frac{1}{4 \alpha^2}\right)$ is easily fullfilled, and noticing that, for any $n \geqslant 3$
$$\tfrac{1}{2} n^2 \log n < \tfrac{n+1}{4} \log \Theta_K < 3 n^2 \log n,$$
we infer that Theorem~\ref{t1} yields the following result.

\begin{coro}
\label{co2}
Let $K$ be an algebraic number field of degree $n \geqslant 3$. Then
$$h_K \mathcal{R}_K < \frac{3w_K}{2}  \left( \frac{2}{\pi} \right)^{r_2} \left(  \frac{\left (\tfrac{1}{4} \log d_K + L_K \right )^{n-1}}{(n-1)!} - \frac{\left (\tfrac{1}{4} \log d_K + \ell_K \right )^{n-2}}{(n-2)!} \right) d_K^{1/2}$$
where $\left. \begin{array}{c} L_K \\ \ell_K \end{array} \right\rbrace:= \tfrac{1}{4} (n^2-1) \log \log d_K + \left\lbrace \begin{array}{c} 3 n^2 \log n \\ \frac{1}{2} n^2 \log n. \end{array} \right.$
\end{coro}

Note that Stirling's bound yields
\begin{multline*}
   h_K \mathcal{R}_K < \frac{3w_K}{2 \sqrt{2 \pi}}  \left( \frac{2}{\pi} \right)^{r_2} \left\lbrace \frac{1}{\sqrt{n-1}} \left( \frac{e \log d_K}{4n-4} + \frac{e L_K}{n-1} \right)^{n-1} \right. \\
   \left. - \frac{e^{-\frac{1}{12(n-2)}}}{\sqrt{n-2}} \left( \frac{e \log d_K}{4n-8} + \frac{e \ell_K}{n-2} \right)^{n-2} \right\rbrace d_K^{1/2}
\end{multline*}
which may be more easily compared to \eqref{eq:Louboutin}.

\subsection*{Acknowledgments}

The author warmly thanks Stephan R. Garcia and Ethan S. Lee for sending him Sunley's thesis \cite{sun} and Lee's preprint \cite{lee22}.


\begin{thebibliography}{9} 
   \bibitem{ayo} \textsc{R. G.~Ayoub}, A mean value theorem for quadratic fields, \emph{Pacific J. Math.} \textbf{8} (1958), 23--27. \\
   \bibitem{bor} \textsc{O.~Bordell\`{e}s}, \emph{Arithmetic Tales, Advanced Edition}, Universitext, Springer, 2020. \\
   \bibitem{cha} \textsc{K.~Chandrasekharan \& R. Narasimhan}, The approximate functional equation for a class of zeta-functions, \emph{Math. Annalen} \textbf{152} (1963), 30--64. \\
   \bibitem{lou1} \textsc{S.~Louboutin}, Explicit bounds for residues of Dedekind zeta functions, values of $L$-functions at $s=1$, and relative class numbers, \emph{J. Number Theory} \textbf{85} (2000), 263--282. \\
   \bibitem{lee22} \textsc{E. Lee}, On the number of integral ideals in a number fields, preprint, 23 pp., 2022. \\
   \bibitem{lou2} \textsc{S.~Louboutin}, Explicit upper bounds for residues of Dedekind zeta functions and values of $L$-functions at $s=1$, and explicit lower bounds for relative class numbers of $CM$-fields, \emph{Canad. J. Math.} \textbf{53} (2001), 1194--1222. \\
   \bibitem{sun} \textsc{J. E. Sunley}, \textit{On the class numbers of totally imaginary quadratic extensions of totally real fields}, Ann Arbor, MI, 1971, Thesis (Ph.D.), University of Maryland, College Park. \\ 
   \bibitem{ten} \textsc{G. Tenenbaum}, \emph{Introduction \`{a} la Th\'{e}orie Analytique et Probabiliste des Nombres}, Belin, 2008 (French). \\
   \bibitem{tit} \textsc{E. M.~Titchmarsh}, \emph{The Theory of the Riemann zeta-function}, Oxford, 1986. Notes by D. R. Heath-Brown. \\
\end{thebibliography}
\end{document}